\documentclass{llncs}
\usepackage{makeidx}  

\usepackage{graphics}
\usepackage[dvips]{graphicx}
\usepackage{amssymb,latexsym,amsmath,stmaryrd}

\newtheorem{defn}{Definition}

\newtheorem{ex}{Example}
\newtheorem{rem}{Remark}
\newtheorem{prop}{Proposition}
\newtheorem{cor}{Corollary}

\begin{document}

\title{Free cyclic submodules in the context of the projective line}

\author{Edyta Bartnicka, Andrzej Matra{\'s}}

\institute{University of Warmia and Mazury in Olsztyn\\ Faculty of Mathematics and Computer Science\\ Poland\\{\ttfamily edytabartnicka@wp.pl, amatras@uwm.edu.pl}}


\maketitle
\setcounter{page}{1}
\begin{abstract}
We discuss the free cyclic submodules over an associative ring $R$ with unity. Special attention is paid to those, which are generated by outliers. This paper describes all orbits of such submodules in the ring of lower triangular $3$x$3$ matrices over a field $F$ under the action of the general linear group. Besides rings with outliers generating free cyclic submodules, there are also rings with outliers generating only torsion cyclic submodules and without any outliers. We give examples of all cases.
\end{abstract}

\begin{keywords} Admissible Pairs - Free Cyclic Submodules - Projective Line - Non-Unimodular Pairs - Outliers
\end{keywords}

\section{Introduction}
In \cite{benz} W. Benz describes classical geometries of M\"{o}bius, Laguerre and Minkowski, using the notion of the projective line over a ring. F.D. Veldkamp in \cite{veld} points out that the assumption of a ring to be of stable rank 2, allows to generalize many properties from classical projective geometry over a field. They both define the projective line by unimodular pairs. \newline
\noindent To admit a wider class of rings, A. Herzer \cite{her} defines a point of the projective line as a cyclic submodule generated by admissible pair. Hence points of $\mathbb{P}(R)$ are elements of the orbit under the action of the $GL_2(R)$. In the present paper we adopt this convention as well.
This approach without any assumptions leads to the existance of points of $\mathbb{P}(R)$ properly contained in another point. A. Blunck and H. Havlicek remark that avoiding this bizarre situation is equivalent to the assumption that a ring is Dedekind-finite, see \cite[Proposition 2.2]{hav1}. \newline
\noindent
H. Havlicek and M. Saniga \cite{hav.san} propose to consider another type of free cyclic submodules, i.e. represented by pairs not contained in any cyclic submodule generated by an unimodular pair (so-called outliers). In this note, we show that the class of non-unimodular free cyclic submodules can be wide. We find four orbits of such submodules in the ring $T_3$ of lower triangular $3$x$3$ matrices over a field $F$ under the action of $GL_2(T_3)$.  On the other hand there are classes of rings without outliers (e.g. semisimple rings, Proposition \ref{semisimple}.) and rings such that outliers generate only torsion submodules (e.g. finite commutative rings, Theorem \ref{finite.com}.). This answers the question posed in \cite{san} about outliers in finite rings. We remark also that there are infinite rings with non-unimodular free cyclic submodules properly contained in unimodular ones (Proposition \ref{PID}). Furthermore, we show that if $R$ is a finite ring and non-unimodular $R(a, b)\subset \ R^2$ is free, then $(a, b)$ is an outlier. In that case, there is no need to check the condition resulting from the definition.\newline\noindent
Using the classification of finite rings, we find all rings up to order $p^4, p-prime$, with outliers generating free cyclic submodules.\newline\noindent
The problem to completely characterize rings with outliers, especially generating free cyclic submodules, is still open.\newline

\section{Preliminaries}
Throughout this paper we shall only consider associative rings with 1 ($1\neq 0$). The group of invertible elements of the ring $R$ will be denoted by $R^*$.
If $R$ is a ring, the expression $R^2$ will mean a left free module over $R$.
If $(a, b)\in R^2$, then the set: $R(a, b)=\{(\alpha a, \alpha b): \alpha \in R\}$ is a {\sl left cyclic submodule of $R^2$}. It is called {\sl free} if the equation $(ra, rb)=(0, 0)$ implies that $r=0$. We assume that $R$ satisfies invariant basis property (IBP) \cite{ibp}. For such rings the basis of cyclic submodule $R(a, b)\subset R^2$ is always of cardinality $1$ and any invertible matrix is in the general linear group $GL_n(R)$ of invertible matrices with entries in $R$.  \\
\noindent
The general linear group $GL_2(R)$ acts in natural way (from the right) on the free left $R$-module $R^2$.
\begin{defn}\label{p.line}{\rm \cite{hav1}}
{\sl The projective line over $R$} is the orbit $$\mathbb{P}(R):=R(1, 0)^{GL_2(R)}$$ of the free cyclic submodule $R(1, 0)$ under the action of $GL_2(R)$.
\end{defn}
In other words, the points of $\mathbb{P}(R)$ are those free cyclic submodules of $R^2$  which possess a free cyclic complement. It provides to introduce admissibility.
\begin{defn}\label{adm}
A pair $(a, b) \in R^2$ is {\sl admissible}, if there exist elements $c, d \in R$ such that $$\left[\begin{array}{cclr}
a&b\\ c&d
\end{array}\right] \in GL_2(R),$$
i.e.  $R(a, b)$ is a free cyclic submodule which has a free cyclic complement.
If $R$ is commutative, then the condition mentioned above is equivalent to $$\det \left[\begin{array}{cclr}
a&b\\ c&d
\end{array}\right] \in R^*.$$
\end{defn}
\noindent
Therefore $\mathbb{P}(R)=\{R(a, b)\subset R^2$, $(a, b)$ admissible\}. As we mentioned before, in earlier definition of the projective line over commutative ring used by W. Benz \cite{benz}, the points of the projective line are cyclic submodules generated by unimodular pairs.
\begin{defn}\label{unim}
A pair $(a, b) \in R^2$ is {\sl right unimodular}, if there exist elements $x, y \in R$ such that $$ax+by=1.$$
\end{defn}
From now on, whenever we will write 'unimodularity', we always mean 'right unimodularity'.
\begin{rem}\label{un.fcs}
Obviously, the admissibility implies the unimodularity and if $(a, b)\in R^2$ is unimodular, then $R(a, b)$ is a free cyclic submodule of $R^2$.
\end{rem}
In contrast to the cyclic submodules generated by admissible pairs, others free cyclic submodules do not have a free cyclic complement in $R^2$.\newline
\\
\noindent The following simple remark describes unimodularity in terms of (right) ideals.
\begin{rem}\label{w.unim}
Let $R$ be a ring and  $a, b \in R$. The following statements are equivalent:
\begin{enumerate}
\item\label{w.unim.1} $aR+bR=R$.
\item\label{w.unim.2} There exist elements $x, y \in R$ such that $ax+by=1$.
\item\label{w.unim.3} There is no proper right ideal $I$ such that $a, b \in I$.
\end{enumerate}
\end{rem}

\begin{proof}
\ref{w.unim.1}.$\Leftrightarrow $\ref{w.unim.2}. See \cite{veld}.
\par\smallskip\noindent
\ref{w.unim.2}. $\Rightarrow$ \ref{w.unim.3}. \quad
Suppose that there exist $x, y \in R$ such that $ax+by=1$ and let $I$ be a right ideal such that $a, b \in I$. Of course, $ax \in I$ for all $x \in R$ and $by \in I$ for all $y \in R$. Consequently $(ax+by) \in I$ for all $x, y \in R$.  Thus $1 \in I$, and therefore $I=R$.
\par\smallskip\noindent
\ref{w.unim.3}. $\Rightarrow$ \ref{w.unim.2}. \quad
Assume that $ax+by\neq 1$ for all $x, y \in R$, then $\{ax+by; x, y \in R\}=aR+bR\neq R$.
$aR, bR$ are ideals, thus $aR+bR$ is an ideal too, and it contains $a,b$. So,  $aR+bR$ is a proper right ideal which contradicts \ref{w.unim.3}.\qed
\end{proof}
In general, cyclic submodules generated by unimodular (resp. admissible) pairs can be also generated by non-unimodular (resp. non-admissible) ones.
In special cases cyclic submodule generated by unimodular (resp. admissible) pair cannot have non-unimodular (resp. non-admissible) representation. \newline
\\
\noindent
We substitute 'admissible' by 'unimodular' in {\rm \cite[Proposition 2.1 (2)]{hav1}} and we get:
\begin{prop} \label{prop.hav} Let  $(x, y)\in R^2$ be unimodular and let $r\in R$. Put $(a, b):=r(x, y).$ Then
\begin{enumerate}
\item $r$ is left invertible if, and only if, $R(x, y)=R(a, b)$.
\item\label{r.right.inv} $r$ is right invertible if, and only if, $(a, b)$ is unimodular.
\end{enumerate}
\end{prop}
\begin{proof}
{\it \ref{r.right.inv}.} Suppose that $(x, y)\in R^2$ is unimodular, then there exist $x', y'\in R$ with $xx'+yy'=1$. Let $r(x, y)=(a, b)$ for some $r\in R$.\newline
$"\Rightarrow"$ \newline
If $r$ is right invertible, then $rs=1$ for some $s\in R$. Hence $(a, b)$ is unimodular:
$$a(x's)+b(y's)=rx(x's)+ry(y's)=
r\Bigl(\bigl((xx')+(yy')\bigr)s\Bigr)=rs=1.$$
$"\Leftarrow"$ \newline
If $(a, b)$ is unimodular, then there exist $a', b' \in R$ with $aa'+bb'=1$, which implies that $r$ has a right inverse:
$$aa'+bb'=(rx)a'+(ry)b'=r(xa'+yb')=1.$$\qed
\end{proof}
\noindent
Rings with the property $ab=1\Rightarrow ba=1$
are called {\em Dedekind-finite}. On account of  the above proposition and of the Proposition 2.1 (2) \cite{hav1} we obtain:
\begin{cor}
If $R$ is Dedekind-finite, then the cyclic submodule $R(a, b)$ generated by an unimodular (resp. admissible) pair do not have non-unimodular (resp. non-admissible) representation.
\end{cor}
\medskip

\noindent
As we know, each admissible pair $(a, b)\in R^2$ is unimodular.
What about the converse implication?
There are examples of rings where unimodularity does not imply admissibility {\rm \cite[Remark 5.1]{hav2}}. However, it is also known that if $R$ is a ring of stable rank 2 (for example, local rings and matrix rings over fields), then admissibility and unimodularity are equivalent and $R$ is Dedekind-finite \cite[Remark 2.4]{hav1}. So finite or commutative rings satisfy this property as well. In case of such rings, the projective line can be described by using unimodularity or admissibility interchangeably. \newline
\\
\noindent
Let us introduce the following temporary notation:\newline
\noindent
$(F)$ \ Any element of the ring $R$ is either invertible or a zero divisor.\newline
\noindent
 It is known that any finite ring satisfies $(F)$. Additionally, if $R$ satisfies $(F)$, then the ring $M_n(R) \ (n\geqslant 1)$ fulfills this condition as well.
\begin{cor}\label{cor.z.hav}
Let $R$ satisfy $(F)$, $(x, y)\in R^2$ be unimodular and $(a, b)=r(x, y)$. Then the following are equivalent:
\begin{enumerate}
\item $r\in R$ is invertible.
\item $R(x, y)=R(a, b)$.
\item $(a, b)$ is unimodular.
\end{enumerate}
\end{cor}
\begin{rem}\label{prop:notfcs}
Let $R$ be a ring and let  $(a, b)\in R^2$. If there exists $(x, y)\in R^2$ and a left zero divisor $r\in R$ such that $(a, b)=r(x, y)$, then $R(a, b)$ is not a free cyclic submodule.
\end{rem}
\begin{proof}
Suppose that the above assumptions are satisfied. Hence there exists nonzero $\alpha \in R$ such that $\alpha r=0$, which yields:
$$\alpha (a, b)=\alpha (rx, ry)=\alpha r(x, y)=(0, 0).$$\qed
\end{proof}
Condition $(F)$ implies that the same free cyclic submodule can be represented by two pairs exactly if they are left-proportional by an invertible element of $R$.
\begin{prop}\label{F}
Let $R$ satisfy $(F)$.
A pair  $(a, b)\in R^2$ generating a free cyclic submodule is contained solely in free cyclic submodules.
\end{prop}
\begin{proof}
Assume that there exist $(x, y)\in R^2$ and $r\in R$ such that $(a, b)=r(x, y)$. According to Remark \ref{prop:notfcs}. $r$ is an invertible element of $R$, hence $R(a, b)=R(x, y)$.\qed
\end{proof}
\noindent
The next class of cyclic submodules, which can be considered in the context of the projective line, is the one proposed by H. Havlicek and M. Saniga in \cite{hav.san}.
\begin{defn}\label{out}
{\rm \cite[Definition 9]{san}} A pair which is not contained in any cyclic submodule generated by an unimodular pair is called {\sl an outlier}.
\end{defn}
In the last section, will be needed one more concept. Recall that the monomorphism $f: {M'}\longrightarrow M$ 
 is called {\sl split} if there exists $g: M\longrightarrow {M'}$ such that $g\circ f=1_{M'}$. In other words, sequence of left modules $0\longrightarrow  {M'}\longrightarrow M$ is split.

\section{Free cyclic submodules generated by non-unimodular pairs}
For some rings there are free cyclic submodules which are generated only by non-unimodular pairs. We establish some connections between outliers and non-principal ideals, \cite{san}.

\begin{prop}\label{suma.g}
Let $(a, b)\in R^2$ be non-unimodular. If $aR+bR$ is a non-principal right ideal, then $(a, b)$ is an outlier.
\end{prop}
\begin{proof}
Assume that $(a, b)$ is not an outlier. By Definition \ref{out}.  there exist $\alpha \in R$ and an unimodular pair $(x, y)\in R^2$ such that $(a, b)=\alpha(x, y)$. Hence $aR+bR=\alpha xR+\alpha yR=\alpha (xR+yR)=\alpha R$, which completes the proof.  \qed
\end{proof}
\begin{cor}\label{out.g}{\rm \cite[Theorem 13.]{san}}
Let $(a, b)\in R^2$ be non-unimodular. Then $(a, b)$ is an outlier, if one of the following conditions is satisfied:
\begin{enumerate}
\item \label{nieg} Does not exist a principal proper right ideal $\alpha R$ such that $a, b\in \alpha R$.
\item \label{g} $aR+bR\varsubsetneq \alpha R$ for all principal proper right ideals $\alpha R$ such that $a, b\in \alpha R$.
\end{enumerate}
\end{cor}

\begin{theorem}\label{id.g}
Let $R$ satisfy $(F)$.\begin{enumerate}
\item\label{g.niew} If there exists a principal proper right ideal $\alpha R$ containing $a$ and $b$, then $R(a, b)$ is a torsion cyclic submodule.
\item\label{out.sk}
If $R(a, b)\subset R^2$ is non-unimodular and free, then $(a, b)$ is an outlier.
\end{enumerate}
\end{theorem}
\begin{proof}
\ref{g.niew}. Suppose that the above assumptions are satisfied. Then there exist $c, d \in R$ such that $a=\alpha c, b=\alpha d$ and nonzero $r\in R$ with $r\alpha =0$. We thus get $r(a, b)=r(\alpha c, \alpha d)=(0, 0)$, which is our claim.\newline
\ref{out.sk}.
If $R(a, b)$ is a free cyclic submodule, then $r\in R$ is invertible for all $(x, y)\in R^2$ such that $(a, b)=r(x, y)$, which follows from Remark \ref{prop:notfcs}. Hence $(a, b)$ is an outlier.\qed
\end{proof}
\begin{ex} \label{gen.niegen}
Consider the ring $R=\{\left[\begin{array}{cclr}
a&0&0 \\ b&a&0 \\c&0&d
\end{array}\right]; a, b, c, d\in GF(2)\}$. Let $$\left[\begin{array}{cclr}
1&0&0 \\ 1&1&0 \\1&0&0
\end{array}\right]=A, \ \left[\begin{array}{cclr}
1&0&0 \\ 1&1&0 \\0&0&0
\end{array}\right]=B, \ \left[\begin{array}{cclr}
1&0&0 \\ 0&1&0 \\1&0&0
\end{array}\right]=C, \ \left[\begin{array}{cclr}
1&0&0 \\ 0&1&0 \\0&0&0
\end{array}\right]=D,$$$$
\left[\begin{array}{cclr}
0&0&0 \\ 1&0&0 \\1&0&0
\end{array}\right]=I, \ \left[\begin{array}{cclr}
0&0&0 \\1&0&0 \\0&0&0
\end{array}\right]=J, \ \left[\begin{array}{cclr}
0&0&0 \\ 0&0&0 \\1&0&0
\end{array}\right]=K, \ \left[\begin{array}{cclr}
0&0&0 \\ 0&0&0 \\0&0&0
\end{array}\right]=0.$$
\noindent
There are exactly two right ideals which are not principal: \ $I_1=\{0, I, J, K\}, \ I_2=\{0, A, B, C, D, I, J, K\}.$ By Proposition \ref{suma.g}. pairs of matrices which are generators of ideals $I_1, I_2$, are outliers. For example,  {\rm gen}$(I, A)=$ {\rm gen}$(K, A)=$ {\rm gen}$(A, D)=$ {\rm gen}$(A, B)=I_2$. An easy calculation shows that every of them generates a free cyclic submodule, for instance,
$$\left[\begin{array}{cclr}
a&0&0 \\ b&a&0 \\c&0&d
\end{array}\right]\Bigl(\left[\begin{array}{cclr}
0&0&0 \\ 1&0&0 \\1&0&0
\end{array}\right], \left[\begin{array}{cclr}
1&0&0 \\ 1&1&0 \\1&0&0
\end{array}\right]\Bigr)=\Bigl(\left[\begin{array}{cclr}
a&0&0 \\ b+a&a&0 \\c+d&0&0
\end{array}\right], \left[\begin{array}{cclr}
0&0&0 \\ a&0&0 \\d&0&0
\end{array}\right]\Bigr)=$$$$=\Bigl(\left[\begin{array}{cclr}
0&0&0 \\ 0&0&0 \\0&0&0
\end{array}\right], \left[\begin{array}{cclr}
0&0&0 \\ 0&0&0 \\0&0&0
\end{array}\right]\Bigr)\Leftrightarrow \left[\begin{array}{cclr}
a&0&0 \\ b&a&0 \\c&0&d
\end{array}\right]=\left[\begin{array}{cclr}
0&0&0 \\ 0&0&0 \\0&0&0
\end{array}\right].$$
There is one principal right ideal $I_3$ contains eight elements, including: $I, J, K$. Notice that the right ideal $IR+JR=\{0, I\}+\{0, J\}=\{0, I, J, K\}\varsubsetneq I_3$. Althought pairs $(I, J)$, $(J, I)$  are outliers, they do not generate free cyclic submodules by Theorem \ref{id.g}. \ref{g.niew}. By the same method follows that pairs $(I, K)$, $(K, I)$ $(K, J)$, $(J, K)$  are outliers too and do not generate free cyclic submodules. In consequence,  there are two different kinds of outliers. {\rm 24} of them generates {\rm 6} free cyclic submodules and {\rm 6} others do not.
\end{ex}
\begin{ex}\label{out.nieout}
Choose the pair $v=\Bigl(\left[\begin{array}{cclr}
0&0&0 \\ 1&0&0 \\1&0&0
\end{array}\right],\left[\begin{array}{cclr}
1&0&0 \\ 1&1&0 \\1&0&0
\end{array}\right]\Bigr)$ over the ring $R$ of matrices $\{\left[\begin{array}{cclr}
a&0&0 \\ b&a&0 \\c&0&d
\end{array}\right], a, b, c, d\in GF(p), p-prime\}$. $v$ is right and left non-unimodular over $R$. The cyclic submodule $Rv$ is free and $v$ is an outlier (from the left). But the cyclic submodule $vR$ is torsion and $v$ is not an outlier (from the right). So the definition of outlier is not symmetric.
\end{ex}
\begin{ex}\label{ternions}
{\rm \cite{havmat}} Consider the ring $T$ of ternions, which is isomorphic to the ring of upper triangular $2${\rm x}$2$ matrices with entries from an arbitrary commutative field $F$. The free cyclic submodules fall into two distinct orbits under the action of the $GL_2(T)$:
$$O_1=T\Bigl(\left[\begin{array}{cclr}
1&0 \\ 0&1
\end{array}\right], \left[\begin{array}{cclr}
0&0 \\ 0&0
\end{array}\right]\Bigr)^{GL_2(T)} \ \ \
O_2=T\Bigl(\left[\begin{array}{cclr}
0&0 \\ 0&1
\end{array}\right], \left[\begin{array}{cclr}
0&1 \\ 0&0
\end{array}\right]\Bigr)^{GL_2(T)}.$$
The first orbit makes up the projective line $\mathbb{P}(T)$, the second one is the orbit of free cyclic submodules generated by outliers.
\end{ex}
\noindent Let $R=T_3$ be the ring of lower triangular 3x3 matrices with entries from an arbitrary commutative field $F$.
\begin{theorem}\label{orb.tr}
Under the action of the general linear group $GL_2(T_3)$ the free cyclic submodules of $T_3$ fall into $5$ distinct orbits. Pairs generating free cyclic submodules of \ $T_3$ fall into $4+|F|$ distinct orbits with the following representatives and corresponding right ideals:
\begin{enumerate}
\item\label{o.unim}
$\Bigl[\left[\begin{array}{cclr}
1&0&0 \\ 0&1&0 \\0&0&1
\end{array}\right],\left[\begin{array}{cclr}
0&0&0 \\ 0&0&0 \\0&0&0
\end{array}\right] \Bigr],\qquad \qquad I_1=\Big\{\left[\begin{array}{cclr}
a&0&0 \\ b&c&0 \\d&e&f
\end{array}\right]: a, b, c, d, e, f\in F\Bigr\};$
\item\label{o.o.1}
$\Bigl[\left[\begin{array}{cclr}
1&0&0 \\ 0&1&0 \\0&0&0
\end{array}\right],\left[\begin{array}{cclr}
0&0&0 \\ 0&0&0 \\0&1&0
\end{array}\right] \Bigr],\qquad \qquad I_2=\Big\{\left[\begin{array}{cclr}
a&0&0 \\ b&c&0 \\d&e&0
\end{array}\right]: a, b, c, d, e\in F\Bigr\};$
\item\label{o.o.2}
$\{\Bigl[\left[\begin{array}{cclr}
1&0&0 \\ 0&1&0 \\0&e&0
\end{array}\right],\left[\begin{array}{cclr}
0&0&0 \\ 0&0&0 \\1&0&0
\end{array}\right] \Bigr], e\in F\},\qquad \qquad I_3=\Big\{\left[\begin{array}{cclr}
a&0&0 \\ b&c&0 \\d&ec&0
\end{array}\right]: a, b, c, d\in F\Bigr\};$
\item\label{o.o.3}
$\Bigl[\left[\begin{array}{cclr}
1&0&0 \\ 0&0&0 \\0&1&0
\end{array}\right],\left[\begin{array}{cclr}
0&0&0 \\ 1&0&0 \\0&0&0
\end{array}\right] \Bigr],\qquad \qquad I_4=\Big\{\left[\begin{array}{cclr}
a&0&0 \\ b&0&0 \\d&e&0
\end{array}\right]: a, b, d, e\in F\Bigr\};$
\item\label{o.o.4}
$\Bigl[\left[\begin{array}{cclr}
1&0&0 \\ 0&0&0 \\0&0&1
\end{array}\right],\left[\begin{array}{cclr}
0&0&0 \\ 1&0&0 \\0&0&0
\end{array}\right] \Bigr],\qquad \qquad I_5=\Big\{\left[\begin{array}{cclr}
a&0&0 \\ b&0&0 \\d&e&f
\end{array}\right]: a, b, d, e, f\in F\Bigr\}.$
\end{enumerate}
\end{theorem}
\begin{proof}
Clearly, all unimodular pairs are in the orbit
 $\Bigl[\left[\begin{array}{cclr}
1&0&0 \\ 0&1&0 \\0&0&1
\end{array}\right],\left[\begin{array}{cclr}
0&0&0 \\ 0&0&0 \\0&0&0
\end{array}\right] \Bigr]^{GL_2(T_3)}.$
Let now $\Bigl[\left[\begin{array}{cclr}
a&0&0 \\ b&c&0 \\d&e&f
\end{array}\right],\left[\begin{array}{cclr}
a'&0&0 \\ b'&c'&0 \\d'&e'&f'
\end{array}\right] \Bigr], a, b, c, d, e, f, a', b', c', d', e', f' \in F$:
\begin{enumerate}
\item \label{non-uni} be non-unimodular;
\item \label{free} generate free cyclic submodule.
\end{enumerate}
We obtain from \ref{free}. that $a\neq 0$ or $a'\neq 0$, and then we assume $a\neq 0$. Multiplying by the invertible matrix $$\left[\begin{array}{cclr}
\left[\begin{array}{cclr}
a^{-1}&0&0 \\
0&1&0 \\
0&0&1 \end{array}\right]
&\left[\begin{array}{cclr}
a^{-1}a'&0&0 \\
0&0&0 \\
0&0&0 \end{array}\right] \\ 0&I
\end{array}\right]$$
gives a pair $\Bigl[\left[\begin{array}{cclr}
1&0&0 \\ b_1&c&0 \\d_1&e&f
\end{array}\right],\left[\begin{array}{cclr}
0&0&0 \\ b_1'&c'&0 \\d_1'&e'&f'
\end{array}\right] \Bigr]$.
We consider now all possibilities:\newline \\
Case 1. $c\neq 0$. The result of multiplication $$\Bigl[\left[\begin{array}{cclr}
1&0&0 \\ b_1&c&0 \\d_1&e&f
\end{array}\right],\left[\begin{array}{cclr}
0&0&0 \\ b_1'&c'&0 \\d_1'&e'&f'
\end{array}\right] \Bigr] \left[\begin{array}{cclr}
\left[\begin{array}{cclr}
1&0&0 \\
0&c^{-1}&0 \\
0&0&1 \end{array}\right]
&\left[\begin{array}{cclr}
0&0&0 \\
0&-c^{-1}c'&0 \\
0&0&0 \end{array}\right] \\ 0&I
\end{array}\right]$$
is a pair $\Bigl[\left[\begin{array}{cclr}
1&0&0 \\ b_1&1&0 \\d_1&e_1&f
\end{array}\right],\left[\begin{array}{cclr}
0&0&0 \\ b_1'&0&0 \\d_1'&e_1'&f'
\end{array}\right] \Bigr]$. From \ref{non-uni}. we get $f=f'=0$. \newline
Case 1.1. $b_1'\neq 0$. We multiply $\Bigl[\left[\begin{array}{cclr}
1&0&0 \\ b_1&1&0 \\d_1&e_1&0
\end{array}\right],\left[\begin{array}{cclr}
0&0&0 \\ b_1'&0&0 \\d_1'&e_1'&0
\end{array}\right] \Bigr]$ by the invertible matrix
$$\left[\begin{array}{cclr}
I&0 \\
\left[\begin{array}{cclr}
-{b'_1}^{-1}b_1&0&0 \\
0&0&0 \\
0&0&0 \end{array}\right]
&\left[\begin{array}{cclr}
{b'_1}^{-1}&0&0 \\
0&1&0 \\
0&0&1 \end{array}\right]
\end{array}\right],$$
which gives a pair $\Bigl[\left[\begin{array}{cclr}
1&0&0 \\ 0&1&0 \\d_2&e_1&0
\end{array}\right],\left[\begin{array}{cclr}
0&0&0 \\ 1&0&0 \\d_2'&e_1'&0
\end{array}\right] \Bigr]$. We know from \ref{free}. that $e_1\neq d_2'$.\newline \\
Case 1.1.1. $e_1'\neq 0$. We multiply again last pair by the invertible matrix $$\left[\begin{array}{cclr}
I&\left[\begin{array}{cclr}
0&0&0 \\
-1&0&0 \\
0&0&0 \end{array}\right] \\
\left[\begin{array}{cclr}
0&0&0 \\
-{e_1'}^{-1}d_2&-{e_1'}^{-1}e_1&0 \\
0&0&0 \end{array}\right]
&\left[\begin{array}{cclr}
1&0&0 \\
{e_1'}^{-1}(e_1-d_2')&{e_1'}^{-1}&0 \\
0&0&1 \end{array}\right]
\end{array}\right]$$
and finally we obtain representative of the orbit: $\Bigl[\left[\begin{array}{cclr}
1&0&0 \\ 0&1&0 \\0&0&0
\end{array}\right],\left[\begin{array}{cclr}
0&0&0 \\ 0&0&0 \\0&1&0
\end{array}\right] \Bigr]$.\newline
In the same manner (considering other cases and multiplying by invertible matrices) we get all orbits of pairs generating free cyclic submodules of $T_3$.\newline
\noindent
It is easy to check that they are distinct, i.e. there is no any invertible matrix that converts one orbit to another. Multiplication (from the left) representatives of orbits by the invertible elements of $T_3$, follows immediately that free cyclic submodules generated by pairs from point 3. are in the same orbit:
 $$\left[\begin{array}{cclr}
1&0&0 \\ 0&1&0 \\0&-e&1
\end{array}\right]\Bigl[\left[\begin{array}{cclr}
1&0&0 \\ 0&1&0 \\0&e&0
\end{array}\right],\left[\begin{array}{cclr}
0&0&0 \\ 0&0&0 \\1&0&0
\end{array}\right] \Bigr]=\Bigl[\left[\begin{array}{cclr}
1&0&0 \\ 0&1&0 \\0&0&0
\end{array}\right],\left[\begin{array}{cclr}
0&0&0 \\ 0&0&0 \\1&0&0
\end{array}\right] \Bigr].$$\qed
\end{proof}
\begin{cor}\label{pair.orbits}
Two pairs $(x, y), (w, z) \in {T_3}^2$ generating free cyclic submodules are in the same $GL_2(T_3)$-orbit if, and only if, the right ideals generated by $x, y \in T_3$ and by $w, z \in T_3$ coincide.
\end{cor}
\begin{proof}
Let $I_{(x, y)}$ denote the right ideal of $T_3$ which is generated by $x$ and $y$. \newline
$"\Rightarrow "$\newline
If pairs $(x, y), (w, z) \in {T_3}^2$ are in the same $GL_2(T_3)$-orbit, then there exists a matrix $A\in GL_2(T_3)$ such that $(x, y)A=(w, z)$. This gives $I_{(w, z)}\subseteq I_{(x, y)}$. \newline
Next we multiply last equation by $A^{-1}$, which yields $(x, y)=(w, z)A^{-1}$, and, in consequence $I_{(x, y)}\subseteq I_{(w, z)}$. The result is $I_{(x, y)}= I_{(w, z)}$.\newline
$"\Leftarrow"$\newline
This is straightforward from Theorem \ref{orb.tr}. \qed
\end{proof}
\noindent
In case of  rings without $(F)$ there are also non-unimodular free cyclic submodules, that are not generated by outliers.
\begin{prop}\label{PID}
If $R$ is a (commutative) PID, then non-unimodular free cyclic submodules are generated by non-outliers.
\end{prop}
\begin{proof}
We use the following equivalent characterization of a proper (commutative) PID:
\begin{itemize}
\item prime ideals are maximal if they are nonzero;
\item prime ideals are principal;
\item $\gcd(a, b)=1\Rightarrow$ gen$(a, b)=1$ for any $a, b\in R$;
\item $R$ is Bezout.
\end{itemize}
Suppopse that $(a, b)\in {R^2}$ is non-unimodular. If $\gcd(a, b)=d$, then $a=dr_1, b=dr_2$ with $\gcd(r_1, r_2)=1$. Hence $(r_1, r_2)$ is unimodular and $(a, b)\in R(r_1, r_2)$, so $(a, b)$ is non-outlier. \qed
\end{proof}

\section{Rings without non-unimodular free cyclic submodules}
\noindent
There are some rings $R$ such that free cyclic submodules $R(a, b)$ are generated only by admissible pairs  $(a, b)\in R^2$. They all makes up the projective line $\mathbb{P}(R)$, for instance, fields or finite local rings \cite[Theorem 20. 1]{san}. In case of these rings all free cyclic submodules can be written as a projective line: $\mathbb{P}(R)=\{R(1, x), x\in R\}\cup \{R(d, 1), d\in I\} I$ - the maximal ideal of $R$.
\begin{prop}\label{semisimple}
Let $R$ be a semisimple ring. A cyclic submodule $R(a, b)$ is free if, and only if $(a, b)\in \ R^2$ is admissible.
\end{prop}
\begin{proof} It follows from the fact that any monomorphism $0\longrightarrow  R(a, b)\longrightarrow R^2$ is split. \cite[Corollary 13.10.]{and.ful}.\qed
\end{proof}
\noindent
Another class of rings without non-unimodular free cyclic submodules are finite principal ideal rings. \cite[Theorem 23]{san}.
\begin{theorem}{\rm\cite[4, VI.2]{mcd}}\label{product.local}
Let $R$ be a commutative finite ring. There exist local rings $R_1, R_2, ... , R_n$ such that $$R=R_1\times R_2\times ... \times R_n.$$
\end{theorem}
\begin{theorem}\label{direct.product}
Let $R$ be a direct product of rings $R_1, R_2, ... , R_n$.
\begin{enumerate}
\item\label{product.unim} A pair $\big((a_1, a_2, ... , a_n), (b_1, b_2, ... , b_n)\big)\in R^2$ is unimodular if, and only if, pairs $$(a_1, b_1)\in{R^2_1}, \ (a_2, b_2)\in {R^2_2}, \ ... \ , \ (a_n, b_n)\in {R^2_n}$$ are unimodular.
\item\label{product.adm} A pair $\big((a_1, a_2, ... , a_n), (b_1, b_2, ... , b_n)\big)\in R^2$ is admissible if, and only if, pairs $$(a_1, b_1)\in{R^2_1}, \ (a_2, b_2)\in {R^2_2}, \ ... \ , \ (a_n, b_n)\in {R^2_n}$$ are admissible.
\item\label{product.out} A pair $\big((a_1, a_2, ... , a_n), (b_1, b_2, ... , b_n)\big) \in R^2$ is an outlier if, and  only if, there exists $i\in \{1, 2, ... , n\}$ such that $(a_i, b_i)\in {R^2_i}$ is an outlier.
\item\label{product.fcs} $R\big((a_1, a_2, ... , a_n), (b_1, b_2, ... , b_n)\big)$ is a free cyclic submodule of $R^2$ if, and only if, $$R_1(a_1, b_1), \ R_2(a_2, b_2), \ ... \ , \ R_n(a_n, b_n)$$ are free cyclic submodules of $R^2_1, R^2_2, ... , R^2_n$.
\end{enumerate}
\end{theorem}
\begin{proof} We give the proof only for the point \ref{product.out}., others are simple consequences of definitions.\\
\noindent
Assume that $(a_i, b_i)\in {R^2_i}$ are not outliers for all $i\in \{1, 2, ... , n\}$. Equivalently, we can say that there exist unimodular pairs $(x_i, y_i)\in {R^2_i}$ and $r_i\in R$ such that
\medskip \\
\centerline {$(a_i, b_i)=r_i(x_i, y_i)$ for all $i\in \{1, 2, ... , n\}.$}
We can write:
$$\big((a_1, a_2, ... , a_n), (b_1, b_2, ... , b_n)\big)=\big((r_1x_1, r_2x_2, ... , r_nx_n), (r_1y_1, r_2y_2, ... , r_ny_n)\big)=$$$$=(r_1, r_2, ... , r_n)\big((x_1, x_2, ... , x_n), (y_1, y_2, ... , y_n)\big).$$
In the light of the point 1.  $\big((x_1, x_2, ... , x_n), (y_1, y_2, ... , y_n)\big)$ is unimodular. According to Definition \ref{out}. $\big((a_1, a_2, ... , a_n), (b_1, b_2, ... , b_n)\big)$ is not an outlier. \qed
\end{proof}
\begin{theorem}\label{finite.com}
If $R$ is a commutative finite ring, then outliers do not generate free cyclic submodules.
\end{theorem}
\begin{proof}
Let $R$ be a commutative finite ring. In the light of Theorem \ref{product.local}. $R$ is a direct product of local rings $R_1, R_2, ... , R_n$. Write $a_i, b_i\in R_i$ for all $i\in \{1, 2, ... , n\}$.\newline
 Equivalent formulation of the above theorem is now: $R\big((a_1, a_2, ... , a_n), (b_1, b_2, ... , b_n)\big)$ is a free cyclic submodule of $R^2$ if, and only if, $\big((a_1, a_2, ... , a_n), (b_1, b_2, ... , b_n)\big)$ is unimodular.\medskip\\
$"\Rightarrow"$ \quad Assume that  $R\big((a_1, a_2, ... , a_n), (b_1, b_2, ... , b_n)\big)$ is a free cyclic submodule of $R^2$. According to Theorem \ref{direct.product}. \ref{product.fcs}. $R_i(a_i, b_i)$ is a free cyclic submodule of $R^2_i$ for all $i\in \{1, 2, ... , n\}$. As we know, $(a_i, b_i)$ is unimodular for all $i\in \{1, 2, ... , n\}$. Theorem \ref{direct.product}. \ref{product.unim}. now yields $\big((a_1, a_2, ... , a_n), (b_1, b_2, ... , b_n)\big)$ is unimodular.\medskip\\
$"\Leftarrow"$ \quad Follows from Remark \ref{un.fcs}.\qed
\end{proof}
\begin{cor}
If $R$ is a commutative finite ring, then all free cyclic submodules make up the projective line $\mathbb{P}(R)$.
\end{cor}
\begin{ex}
Let us consider the finite noncommutative ring $R$ of characteristic $p^2, p-prime$. The additive group of $R$ is equal to $R^+=Z_{p^2}\oplus Z_p\oplus Z_p$ with a basis $\{1, t, y\}$. The multiplication in the ring $R$ is uniquely determined by the relations $t^2=0, y^2=y, ty=0, yt=t$, {\rm \cite{p4}}.\newline
\noindent We have $(1-t-y)(r+st+hy)+t(r'+s't+h'y)=r+(-r+r')t-ry$ for some $0\leqslant r, r'\leqslant p^2-1, \ 0\leqslant s, h, s', h'\leqslant p-1$, hence the pair $(1-t-y, t)$ is non-unimodular. It is easily seen that $R(1-t-y, t)$ is free. On account of Theorem \ref{id.g}. \ref{out.sk}. $(1-t-y, t)$ is an outlier. $R$ is an example of a ring non-embeddable into the ring of matrices over $GF(p^2)$, {\rm \cite{sych}}.
\end{ex}
\noindent
Now we are able to describe all finite rings up to order $p^4, p-prime$, with outliers generating free cyclic submodules. Since any finite ring with identity is isomorphic to direct sum of rings with identity of prime power order (see \cite{mcd}) and according to Theorems \ref{direct.product}. and \ref{finite.com}., we may restrict ourselves to the study noncommutative indecomposable rings up to order $p^4, p-prime$. By direct calculation, taking into account classification theorems (\cite{p3}, \cite{p4}) we find that there are exactly four such rings for any $p$:
\begin{itemize}
\item of order $p^3$: the ring of ternions over $GF(p)$;
\item of order $p^4$:
\begin{itemize}
\item with characteristic $p$: $$\{\left[\begin{array}{cclr}
a&0&0 \\ b&a&0 \\c&0&d
\end{array}\right], a, b, c, d\in GF(p)\}; \ \ \ \ \{\left[\begin{array}{cclr}
a&c&d \\ 0&b&0 \\0&0&b
\end{array}\right], a, b, c, d\in GF(p)\};$$
\item with characteristic $p^2$ - the ring from the Example 4.
\end{itemize}
\end{itemize}
\noindent {\bf Acknowledgement}. The authors wish to express their thanks to Stanis{\l}aw Drozda for computer support of our research.

\end{document}